\documentclass{amsart}
\usepackage{amsmath,amsfonts,amsthm,amssymb,verbatim,enumerate,graphicx}
\usepackage[flushmargin]{footmisc}
\newcommand{\comments}[1]{}
\newtheorem{theorem}{Theorem}[section]

\newtheorem{lemma}{Lemma}[section]
\newtheorem{corollary}{Corollary}[section]

\newtheorem{remarkk}{Remark}
\numberwithin{equation}{section}
\def \isnatural {\in\mathbb{N}}

\def \iscomplex {\in\mathbb{C}}

\def \spws {spiders' webs}
\def \classb {$\mathcal{B}$}
\newcommand{\tef}{transcendental entire function}
\newcommand\qfor{\quad\text{for }}
\newcommand\Real{\operatorname{Re}}
\newcommand\Imag{\operatorname{Im}}
\begin{document}
%
% OR THIS
%
\title[Functions for which the fast escaping set has dimension two.]{Functions of genus zero for which the fast escaping set has Hausdorff dimension two.}
\author{D. J. Sixsmith}
\address{Department of Mathematics and Statistics \\
	 The Open University \\
   Walton Hall\\
   Milton Keynes MK7 6AA\\
   UK}
\email{david.sixsmith@open.ac.uk}
%%%%%%%%%%%%%
%
% ABSTRACT
%
%%%%%%%%%%%%%
\begin{abstract}
We study a family of {\tef}s of genus zero, for which all of the zeros lie within a closed sector strictly smaller than a half-plane. In general these functions lie outside the Eremenko-Lyubich class. We show that for functions in this family the fast escaping set has Hausdorff dimension equal to two.
\end{abstract}
\maketitle
%
%%%%%%%%%%%%%
%
% INTRO
%
%%%%%%%%%%%%%
\let\thefootnote\relax
\let\thefootnote\relax\footnote{The author was supported by Engineering and Physical Sciences Research Council grant EP/J022160/1.}
\section{Introduction}
Suppose that $f:\mathbb{C}\rightarrow\mathbb{C}$ is a {\tef}. The \itshape Fatou set \normalfont $F(f)$ is defined as the set of points $z\iscomplex$ such that $(f^n)_{n\isnatural}$ is a normal family in a neighbourhood of $z$. Since $F(f)$ is open, it consists of at most countably many connected components which are called \itshape Fatou components\normalfont. The \itshape Julia set \normalfont $J(f)$ is the complement in $\mathbb{C}$ of $F(f)$. An introduction to the properties of these sets was given in \cite{MR1216719}.

For a general {\tef} the \itshape escaping set \normalfont $$I(f) = \{z : f^n(z)\rightarrow\infty\text{ as }n\rightarrow\infty\}$$ was studied in \cite{MR1102727}. The set $I(f)$ now plays a key role in the study of complex dynamics, particularly with regard to the major open question, first asked in \cite{MR1102727}, of whether $I(f)$ necessarily has no bounded components. This question is known as \itshape Eremenko's conjecture\normalfont.

The \itshape fast escaping set\normalfont, $A(f) \subset I(f)$, was introduced in \cite{MR1684251}, and was defined in \cite{Rippon01102012} by
\begin{equation*}
A(f) = \{z : \text{there exists } \ell \isnatural \text{ such that } |f^{n+\ell}(z)| \geq M^n(R,f), \text{ for } n \isnatural\}.
\end{equation*}
Here, the \itshape maximum modulus function \normalfont $M(r,f) = \max_{|z|=r} |f(z)|,$ for $r \geq 0,$ $M^n(r,f)$ denotes repeated iteration of $M(r,f)$ with respect to the variable $r$, and $R > 0$ is such that $M(r,f) > r$, for $r \geq R$. The set $A(f)$ also now plays a key role in the study of complex dynamics. One reason for this is that $A(f)$ has no bounded components \cite[Theorem 1]{MR2117213}, and this provides a partial answer to Eremenko's conjecture. We refer to \cite{Rippon01102012} for a detailed account of the properties of this set.

%
% EL and the B K result
%
The Eremenko-Lyubich class, {\classb}, is defined as the class of {\tef}s for which the set of singular values is bounded. Most existing results on the size of the Julia set of a {\tef}, $f$, concern the case that $f\in\text{\classb}$. A key paper which concerns the size of the Julia set of a {\tef} which may be outside the class {\classb} is that of Bergweiler and Karpi{\'n}ska \cite{MR2609307}. Their principle result is as follows. For $E \subset \mathbb{C}$, we let $\dim_H E$ denote the Hausdorff dimension of $E$, and refer to \cite{falconer} for a definition of Hausdorff dimension.
\begin{theorem}
\label{BKth}
Suppose that $f$ is a {\tef} and that there exist $A, B, C, r_1 > 1$ such that 
\begin{equation}
\label{BKeq}
A \log M(r, f) \leq \log M(Cr, f) \leq B \log M(r, f), \qfor r\geq r_1.
\end{equation}
Then $\dim_H J(f) \cap I(f) = 2$.
\end{theorem}
The techniques used to prove Theorem~\ref{BKth} are very different to those used in earlier papers on dimension. Roughly speaking, the key idea is to construct a set in which the logarithmic derivative is large, and approximately constant in certain small discs. Bergweiler and Karpi{\'n}ska showed that the existence of this set follows from (\ref{BKeq}), which imposes a very strong regularity condition on the growth of the function. 

The aim of this paper is to show that there is a large family of functions, which need not satisfy this regularity condition, for which a set with somewhat similar properties may be constructed. We then show that for functions in this family the \itshape fast \normalfont escaping set has Hausdorff dimension equal to $2$.
 
%
% Or result
%
In particular, we study a family of {\tef}s of \itshape genus \normalfont zero, for which all of the zeros lie within a closed sector strictly smaller than a half-plane; we refer to \cite[p.196]{MR0188405} for the definition of genus. In general these functions lie outside the class {\classb}. 

Our main result is as follows. Here, and throughout the paper, we denote by $\arg(z)$ the principle argument of $z$, for $z \ne 0$; in other words $\arg(z)\in(-\pi,\pi]$.
\begin{theorem}
\label{maintheo}
Suppose that $f$ is a {\tef} of the form
\begin{equation}
\label{fdef}
f(z) = c z^q \prod_{n=1}^\infty \left(1 + \frac{z}{a_n}\right), \text{ where } c\ne0, \ q\in\{0,1,\cdots\}, \text{ and } 0<|a_1|\leq |a_2|\leq\cdots. 
\end{equation}
%where $c\ne0, \ q\in\{0,1,\cdots\}$, and $0<|a_1|\leq |a_2|\leq\cdots$. 
Suppose that there exist positive constants $\theta_1$ and $\theta_2$, such that $0 \leq \theta_2 - \theta_1 < \pi$, and also $N_0\isnatural$ such that
%
%$\theta \in [0, \pi/2), \ \phi \in (-\pi+\theta, \pi-\theta]$ and $N_0\isnatural$ such that 
\begin{equation}
\label{angleconditiona}
\arg(a_n) \in [\theta_1, \theta_2] \mod 2\pi , \qfor n\geq N_0.
%\arg(a_n) \in [\phi - \theta, \phi + \theta], \qfor n\geq N_0.
\end{equation}
Suppose also that $f$ has no multiply connected Fatou components. Then, $$\dim_H J(f)~\cap~A(f) = 2.$$
\end{theorem}
Functions for which (\ref{BKeq}) is satisfied have no multiply connected Fatou components \cite[Theorem 4.5]{MR2609307}. It is well-known that there are functions of the form (\ref{fdef}), such that (\ref{angleconditiona}) is satisfied and which have a multiply connected Fatou component. However, if $U$ is a multiply connected Fatou component of $f$, then $\overline{U}~\subset~A(f)$; see \cite[Theorem 2]{MR2117213} and \cite[Theorem 4.4]{Rippon01102012}. We deduce the following corollary of Theorem~\ref{maintheo}.
\begin{corollary}
Suppose that $f$ is a {\tef} of the form (\ref{fdef}), such that (\ref{angleconditiona}) is satisfied. Then $\dim_H A(f) = 2$.
\end{corollary}

\begin{remarkk}\normalfont
An example of a {\tef}, $f$, of the form (\ref{fdef}), such that (\ref{angleconditiona}) is satisfied, $A(f) \cap F(f) \ne \emptyset$ and $f$ has no multiply connected Fatou components was given in \cite{MR2959920}. This function is outside the class {\classb}. 
\end{remarkk}
\begin{remarkk}\normalfont
%There are functions in the class {\classb} of the form (\ref{fdef}), and such that (\ref{angleconditiona}) is satisfied. 
If $f \in \mathcal{B}$, then $f$ does not have a multiply connected Fatou component \cite[Proposition 3]{MR1196102}. It follows that if $f\in\mathcal{B}$ is of the form (\ref{fdef}), and such that (\ref{angleconditiona}) is satisfied, then $\dim_H J(f) \cap A(f) = 2.$ An example of such a function \cite[p.75]{standards} is $$\cos \sqrt{z} = \prod_{n=1}^\infty \left(1 - \frac{z}{\pi^2 (n-\frac{1}{2})^2}\right).$$ 
\end{remarkk}

%
% R and S
%
\begin{remarkk}\normalfont
Rippon and Stallard \cite{negativerealzeros} also studied a family of functions of genus zero. In particular, they studied the dynamics of {\tef}s of the form
\begin{equation*}
f(z) = c z^q \prod_{n=1}^\infty \left(1 + \frac{z}{\alpha_n}\right), \text{ where } c\in\mathbb{R}\backslash\{0\}, \ q\in\{0,1,\cdots\}, \ \text{and } 0<\alpha_1\leq \alpha_2\leq\cdots.
\end{equation*}
It follows from Theorem~\ref{maintheo}, together with results in \cite{negativerealzeros}, that if $f$ is a function of this form, which also satisfies an additional condition on the minimum modulus of $f$, then $J(f)$ and $J(f) \cap I(f)$ are {\spws} of Hausdorff dimension $2$; we refer to \cite{negativerealzeros} for background and definitions.
\end{remarkk}

%
% structure
%
The structure of this paper is as follows. First, in Section~\ref{Sdefs}, we give a number of definitions which are used throughout the paper. In Section~\ref{Sover}, we give a sequence of three new lemmas which are required for the proof of Theorem~\ref{maintheo}. In Section~\ref{Slast} we give the proof of Theorem~\ref{maintheo}. Finally, we prove the three new lemmas in Sections~\ref{prooflemma1}, \ref{Sdomains} and \ref{buildsets} respectively.%  Finally, in Section~\ref{Sspws}, we estimate the size of some {\spws}.
%
%
%%%%%%%%%%%%%
%
%%%%%%%%%%%%%
%
%
\section{Definitions and assumptions}
\label{Sdefs}
In this section we give a number of definitions and assumptions used in the proof of Theorem~\ref{maintheo}, which should be taken to be in place throughout the paper. 

First, composing with a rotation if necessary, we may assume that $\theta_2 \geq 0$ and that $\theta_1 = -\theta_2$. With these assumptions, we observe that $\theta_2 < \pi/2$, and that it follows from (\ref{angleconditiona}) that
\begin{equation}
\label{anglecondition}
|\arg(a_n)| \leq \theta_2, \qfor n\geq N_0.
\end{equation}

We choose values of $\psi$ and $\psi'$ such that $\theta_2 < \psi < \psi' < \pi/2,$ and set $\sigma= \cos \psi'$. We define a function
\begin{equation}
\label{mudef}
\mu(r) = M(\sigma r,f), \qfor r>0.
\end{equation}
Since it is well-known that 
\begin{equation}
\label{itgrows}
\frac{\log M(r,f)}{\log r} \rightarrow\infty \text{ as } r\rightarrow\infty,
\end{equation}
we may choose $r_0>0$ sufficiently large that
\begin{equation}
\label{mugrows}
\mu(r) > r, \qfor r\geq r_0.
\end{equation}

We next define a number of sets. First, define, for $r>0$, %$$S = \{ z : |\arg(z)| \leq \psi - \theta_2\},$$
$$S(r) = \{ z :  |\arg(z)| \leq (\psi - \theta_2), \ r \leq |z| \},$$ and $$T(r) = \{ z :  |\arg(z)| \leq (\psi - \theta_2), \ r \leq |z| \leq 2r \}.$$ 

As in \cite{MR2609307}, we define domains 
\begin{equation}
\label{thePk}
P_\kappa = \{z : |\Real (z)| < 1,|\Imag (z) - 8\pi \kappa| < 3\pi\}, \qfor \kappa \in \{1,2,3\}.
\end{equation}
For each $a$ such that $f(a) > 0$, and for $\kappa \in \{1,2,3\}$, we also define a domain
\begin{equation}
\label{theOmegak}
\Omega_\kappa(a) = \{z : |\Real(z) - \log f(a)| < 1, \ |\Imag (z) - 8\pi \kappa| < 3\pi\},
\end{equation}
and a compact subset of $\Omega_\kappa(a)$
\begin{equation}
\label{theQk}
Q_\kappa(a) = \{z : 0 \leq \Real(z) - \log f(a) \leq \log 2, \ |\Imag (z) - 8\pi \kappa| \leq \psi - \theta_2\}.
\end{equation}
We use the notation for a disc $$B(a, r) = \{ z: |z - a| < r \}, \qfor r > 0.$$ 
%
%
%%%%%%%%%%%%%
%
%
\section{Lemmas required for the proof of Theorem~\ref{maintheo}}
\label{Sover}
In this section we give three preliminary lemmas required for the proof of Theorem~\ref{maintheo}. These are proved in later sections. For ease of comparison we have, where possible, maintained a consistent terminology with that in \cite{MR2609307}.

We first show that, provided that $r>0$ is sufficiently large, we have very good control on the size of the modulus of $f$, and its logarithmic derivative, in $T(r)$.
\begin{lemma}
\label{sizelemma}
Suppose that $f$ is a {\tef} of the form (\ref{fdef}), such that (\ref{angleconditiona}) is satisfied, that $r_0$ is as defined prior to (\ref{mugrows}), and that $\mu$ is the function defined in (\ref{mudef}). Then there exists $r_1\geq r_0$ such that;
\begin{align}
\label{comparabletoM}
&|f(z)| \geq \mu(|z|), \qfor z\in S(r_1); \\
\label{logderivdunchangemuch}
&\left|z\frac{f'(z)}{f(z)}\right| \geq \frac{\sigma^4}{4} \left|w\frac{f'(w)}{f(w)}\right|, \qfor z, w \in T(r), \ r \geq r_1;\\
\label{logderivisnice}
&\left|r\frac{f'(r)}{f(r)}\right| \geq \frac{\sigma^2}{8} \frac{\log M(r,f)}{\log r} > 0, \qfor r \geq r_1;\\ 
\label{lastlogderiv}
&\left|f'(z)\right| \geq \frac{\sigma^6}{64} \frac{\log M(r,f)}{r \log r} |f(z)|, \qfor z \in T(r), \ r \geq r_1.
\end{align}
\end{lemma}
The proof of this lemma is given in Section~\ref{prooflemma1}. We note in these equations a significant difference between the properties our family of functions and the properties of the family considered in \cite{MR2609307}. In \cite{MR2609307}, a set is constructed in which $|f(z)|$ is greater than a fixed power of $|z|$, and $|z f'(z)/f(z)|$ is comparable to $\log M(|z|,f)$. \\

We next show that, for large values of $r>0$, we have that $T(r)$ contains a large number of compact sets each of which is mapped univalently by $f$ onto some $T(r')$, where $r'$ is large compared to $r$; these are the sets $V_{q, r}$ in the statement of the following lemma. Preimages of these sets $V_{q, r}$ are subsequently used to construct a set which lies in $A(f)$ and has Hausdorff dimension $2$.

If $U\subset \mathbb{C}$ is measurable, then we denote the Lebesgue measure of $U$ by area$(U)$. Note that, in the statement of the following lemma, the real valued function $t(r)$ is defined in (\ref{trdef}) below, and may be taken to be small compared to $r$, and the real valued function $m(r)$ is defined in (\ref{mrdef}) below, and may be taken to be large.
\begin{lemma}
\label{domainslemma}
Suppose that $f$ is a {\tef} of the form (\ref{fdef}), such that (\ref{angleconditiona}) is satisfied, and that $r_1$ is as defined in the statement of Lemma~\ref{sizelemma}. Then there exist constants $c_2 > 0$ and $r_2 \geq r_1$ with the following property. For all $r \geq r_2$, there exist points $b_{q, r}$, domains $U_{q, r}$, and connected compact sets $V_{q, r}$, such that the following all hold, for $1 \leq q \leq m(r)$.
\begin{enumerate}[(i)]
\item The discs $B(b_{q, r}, t(r))$ are pairwise disjoint.
\item $f(b_{q, r}) > 0$.
\item $V_{q, r} \subset U_{q, r} \subset B(b_{q, r}, t(r)) \subset B(b_{q, r}, 2t(r)) \subset T(r)$.
\item The function $f$ is bounded away from zero in $U_{q, r}$. Moreover, there is a branch of the logarithm and some $\kappa\in\{1,2,3\}$ such that:
     \begin{enumerate}[(a)]
     \item $U_{q, r}$ is mapped bijectively by $\log f$ onto $\Omega_\kappa(b_{q,r})$;
     \item $V_{q, r}$ is mapped bijectively by $\log f$ onto $Q_\kappa(b_{q,r})$;
     \item $V_{q, r}$ is mapped bijectively by $f$ onto $T(f(b_{q,r}))$.
     \end{enumerate}
\item $\operatorname{area}(V_{q, r}) \geq c_2 t(r)^2$.%, and also that
\end{enumerate}
\end{lemma}
The various sets and points constructed in this lemma are illustrated in Figure~\ref{fig.y}. %Roughly speaking, the set inclusions in part (iii) allow us to control the distortion of an inverse branch of $f$ which can be defined in $B(b_{q, r}, t(r))$. 
The proof of this lemma is given in Section~\ref{Sdomains}.

\begin{figure}[ht]
	\centering
	\includegraphics[width=10cm,height=7cm]{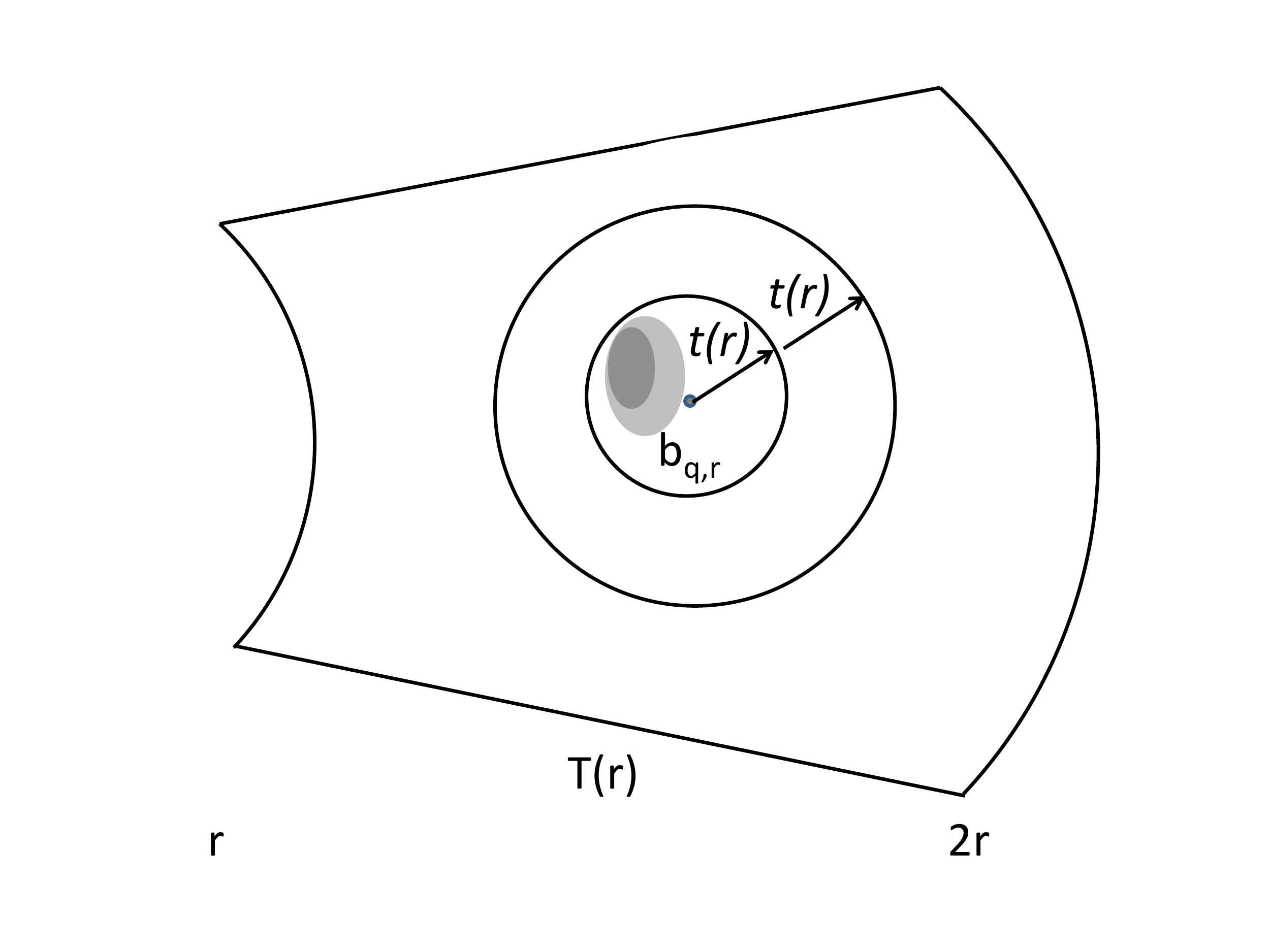}
	\caption{Sets and points constructed in Lemma~\ref{domainslemma}. The domain $U_{q,r}$ is shown lightly shaded, and the compact set $V_{q,r}$ is shown shaded darkly. In this case, to make the sets easy to identify, we have assumed that $m(r) = 1$.}
  \label{fig.y}
\end{figure}

We use a well-known construction of McMullen in order to estimate the Hausdorff dimension of $J(f) \cap A(f)$. For each $n\isnatural$, suppose that $\mathcal{E}_n$ is a finite collection of pairwise disjoint compact subsets of $\mathbb{C}$ such that the following both hold:
\begin{enumerate}[(i)]
\item If $F \in \mathcal{E}_{n+1}$, then there exists a unique $G \in \mathcal{E}_{n}$ such that $F \subset G$;
\item If $G \in \mathcal{E}_{n}$, then there exists at least one $F \in \mathcal{E}_{n+1}$ such that $G \supset F$.
\end{enumerate}
We write 
\begin{equation}
\label{Edef}
E_n = \bigcup_{F\subset \mathcal{E}_n} F, \text{ for } n\isnatural, \quad \text{and} \quad E = \bigcap_{n\isnatural} E_n.
\end{equation}

Our final lemma is as follows. Here, for measurable sets $U$ and $V$, we define $$\operatorname{dens}(U, V) = \frac {\operatorname{area} (U \cap V)}{\operatorname{area}(V)}.$$
\begin{lemma}
\label{Elemma}
Suppose that $f$ is a {\tef} of the form (\ref{fdef}), such that (\ref{angleconditiona}) is satisfied, that $r_2$ is as defined in the statement of Lemma~\ref{domainslemma} and that $\mu$ is the function defined in (\ref{mudef}). Then there exists a sequence of finite collections of pairwise disjoint compact sets, $(\mathcal{E}_n)_{n\isnatural}$, which satisfies conditions (i) and (ii) above, and, with $E$ and $(E_n)_{n\isnatural}$ as defined in (\ref{Edef}), such that 
\begin{equation}
\label{Eisnice}
f^n(z) \in S(r_2), \qfor n\in\{0, 1, \ldots\}, \ z \in E.
\end{equation}
In addition there exist a constant $c_3 > 0$ and a sequence of positive real numbers $(d_n)_{n\isnatural}$, with $d_n \rightarrow 0$ as $n\rightarrow\infty$, such that if $n\isnatural$ and $F\in \mathcal{E}_n$, then 
\begin{equation}
\label{densityofsets}
\operatorname{dens}(E_{n+1},F) \geq c_3,
\end{equation}
\begin{equation}
\label{diamofsets}
\operatorname{diam}F \leq d_n,
\end{equation}
and
\begin{equation}
\label{dneq}
\lim_{n\rightarrow\infty} \frac{n}{|\log d_n|} = 0.
\end{equation}
\end{lemma}
We prove this lemma in Section~\ref{buildsets}. 

%
%%%%%%%%%%%%%%
%
%%%%%%%%%%%%%%
%
\section{Proof of Theorem~\ref{maintheo}}
\label{Slast}
To prove Theorem~\ref{maintheo}, we require the following three lemmas. The first is a key result of McMullen \cite[Proposition 2.2]{MR871679}.
\begin{lemma}
\label{mcmlemma}
Suppose that there exists a sequence of finite collections of pairwise disjoint compact sets, $(\mathcal{E}_n)_{n\isnatural}$, which satisfies conditions (i) and (ii) above, and let $E$ and $(E_n)_{n\isnatural}$ be as defined in (\ref{Edef}). Suppose also that $(\Delta_n)_{n\isnatural}$ and $(d_n)_{n\isnatural}$ are sequences of positive real numbers, with $d_n \rightarrow 0$ as $n\rightarrow\infty$, such that for each $n\isnatural$ and for each $F \in \mathcal{E}_n$, we have $$\operatorname{dens}(E_{n+1}, F) \geq \Delta_n\quad\text{and}\quad\operatorname{diam} F \leq d_n.$$ Then
\begin{equation*}
\dim_H E \geq 2 - \limsup_{n\rightarrow\infty} \frac{\sum_{k=1}^n | \log \Delta_k|}{| \log d_n|}.
\end{equation*}
\end{lemma}
The second, which is a version of \cite[Theorem~2.7]{Rippon01102012}, gives an alternative characterisation of $A(f)$. 
\begin{lemma}
\label{otherA}
Suppose that $f$ is a {\tef} and that $R > 0$ is such that $\mu(r) > r$, for $r \geq R$, where $\mu$ is the function defined in (\ref{mudef}). Then
\begin{equation*}
A(f) = \{z : \text{there exists } \ell \isnatural \text{ such that } |f^{n+\ell}(z)|\geq \mu^n(R), \text{ for } n\isnatural \}.
\end{equation*}
\end{lemma}
The third lemma is \cite[Theorem 3.1]{areapaper}.
\begin{lemma}
\label{LinJulia}
Suppose that $f$ is a {\tef} and that $z~\in~I(f)$. Set $z_n = f^n(z)$, for $n\isnatural$. Suppose that there exist $\lambda > 1$ and $N\geq 0$ such that 
\begin{equation}
\label{orbiteq}
f(z_{n}) \ne 0 \quad\text{ and }\quad \left|z_n \frac{f'(z_n)}{f(z_n)}\right| \geq \lambda, \qfor n\geq N.
\end{equation}
Then either $z$ is in a multiply connected Fatou component of $f$, or $z \in J(f)$.
\end{lemma}

Now we prove Theorem~\ref{maintheo}. Let $E$ be the set defined in the statement of Lemma~\ref{Elemma}. It follows at once from Lemma~\ref{Elemma} and Lemma~\ref{mcmlemma} that $\dim_H E = 2$. %\\ %We show that $E \subset J(f) \cap A(f)$.

We next show that $E \subset A(f)$. We note that it follows from (\ref{comparabletoM}) and (\ref{Eisnice}) that 
\begin{equation}
\label{fbig}
|f^n(z)| \geq \mu^n(|z|) \geq \mu^n(r_0), \qfor n\isnatural, z \in E.
\end{equation}
The fact that $E \subset A(f)$ follows from (\ref{mugrows}) and (\ref{fbig}), by Lemma~\ref{otherA}. 

It remains to show that either $f$ has a multiply connected Fatou component, or $E \subset J(f)$. It follows from (\ref{itgrows}) and (\ref{lastlogderiv}) that there exists $R'>0$ such that 
\begin{equation}
\label{zeq}
\left|z\frac{f'(z)}{f(z)}\right| \geq \frac{\sigma^6}{64} \frac{\log M(r,f)}{\log r} \geq 2, \qfor z \in T(r), \ r \geq R'.
\end{equation} 

Suppose that $z \in E$. Since $z \in A(f)$, there exists $N\isnatural$ such that $|f^n(z)| \geq R'$, for $n\geq N$. Hence, by (\ref{Eisnice}) and (\ref{zeq}), equation (\ref{orbiteq}) holds with $\lambda = 2$. We deduce from Lemma~\ref{LinJulia} that either $z\in J(f)$, or $z$ is in a multiply connected Fatou component of $f$. This completes the proof of the theorem.

%
%%%%%%%%%%%%%%%
%
%%%%%%%%%%%%%%%
%
\section{Proof of Lemma~\ref{sizelemma}}
\label{prooflemma1}
Recalling that $0 < \psi' < \pi/2$ and $\sigma = \cos \psi'$, we observe that, in general, 
\begin{equation}
\label{genleq}
\sigma \left|\zeta\right| \leq \Real (\zeta), \qfor |\arg(\zeta)| \leq \psi'.
\end{equation}
We choose $$r_1 \geq \max\left\{r_0, \frac{2|a_{N_0-1}|}{1-\sigma}, 3^{\frac{3}{2}}\right\}$$ sufficiently large that $M(r_1, f)~>~1$ and
\begin{equation*}
|\arg(z + a_n)| \leq \psi', \qfor z \in S(r_1), \ n\isnatural.
\end{equation*}
We put one additional lower bound on the size of $r_1$; this is after equation (\ref{qwe}).

It follows from (\ref{genleq}), with $\zeta = z + a_n$, that
\begin{equation}
\label{aneweq}
\Real\left(\frac{1}{z+a_n}\right) = \frac{\Real(z + a_n)}{|z+a_n|^2} \geq \frac{\sigma^2}{\Real(z + a_n)}, \qfor z \in S(r_1), \ n\isnatural.
\end{equation}

We first prove (\ref{comparabletoM}). Let $z \in S(r_1)$, and suppose that $w$ is such that $|w| = \sigma|z|$ and $M(\sigma|z|,f)~=~|f(w)|$. Then 
\begin{equation}
\label{mueq}
\mu(|z|) = |f(w)| \leq |c| |z|^q \prod_{n=1}^\infty \left|1 + \frac{w}{a_n}\right|.
\end{equation}

We claim that 
\begin{equation}
\label{claim1}
\left|1 + \frac{w}{a_n}\right| \leq \left|1 + \frac{z}{a_n}\right|, \qfor n\isnatural.
\end{equation}

First suppose that $1 \leq n< N_0$. Since $|z|  \geq r_1$, we have that $$\frac{|z|}{|a_n|} - \frac{|w|}{|a_n|} = (1-\sigma) \frac{|z|}{|a_n|} \geq (1-\sigma) \frac{r_1}{|a_{N_0-1}|} \geq 2.$$ Hence, as required, $$\left|1 + \frac{w}{a_n}\right| \leq \left|\frac{w}{a_n}\right| + 1 \leq \left|\frac{z}{a_n}\right| - 1\leq \left|1 + \frac{z}{a_n}\right|.$$   
On the other hand, suppose that $n\geq N_0$. It follows, by (\ref{genleq}) with $\zeta = \frac{z}{a_n}$, that $$\left|1 + \frac{w}{a_n}\right| \leq 1 + \left|\frac{w}{a_n}\right| = 1 + \sigma\left|\frac{z}{a_n}\right| \leq 1 + \Real\left(\frac{z}{a_n}\right) \leq \left|1 + \frac{z}{a_n}\right|.$$

This completes the proof of (\ref{claim1}), and (\ref{comparabletoM}) follows from (\ref{mueq}) and (\ref{claim1}).

Next we prove (\ref{logderivdunchangemuch}). Suppose that $r \geq r_1$, and that $z, w \in T(r)$. Now
\begin{equation*}
\left|z\frac{f'(z)}{f(z)}\right| = |z|\left| \frac{q}{z} + \sum_{n=1}^\infty \frac{1}{z + a_n} \right| \geq r \left(\Real\left(\frac{q}{z}\right) + \sum_{n=1}^\infty \Real\left(\frac{1}{z + a_n}\right)\right).
\end{equation*}
It follows from (\ref{genleq}), with $\zeta = z$, that $$\Real\left(\frac{1}{z}\right) = \frac{\Real (z)}{|z|^2} \geq \frac{\sigma}{|z|} \geq \frac{\sigma}{2r}.$$ Also, it follows from (\ref{aneweq}) that $$\Real\left(\frac{1}{z+a_n}\right) \geq \frac{\sigma^2}{|z|+\Real (a_n)}\geq \frac{\sigma^2}{2r + \Real (a_n)}, \qfor n\isnatural.$$ We deduce that 
\begin{equation}
\label{f1}
\left|z\frac{f'(z)}{f(z)}\right| \geq r \left(\frac{q\sigma}{2r} + \sigma^2 \sum_{n=1}^\infty \frac{1}{2r + \Real (a_n)}\right).
\end{equation}
We also have, by (\ref{genleq}) with $\zeta = 1/(w + a_n)$, that
\begin{equation*}
\left|w\frac{f'(w)}{f(w)}\right| \leq q + |w| \sum_{n=1}^\infty \frac{1}{|w + a_n|} \leq q + \frac{2r}{\sigma} \sum_{n=1}^\infty \Real \left(\frac{1}{w + a_n}\right).
\end{equation*}
Now, by (\ref{genleq}) with $\zeta = w$, $$\Real \left(\frac{1}{w+a_n}\right) \leq \frac{1}{ \Real(w+a_n)} \leq \frac{1}{r\sigma + \Real (a_n)} \leq \frac{2}{\sigma} \frac{1}{2r + \Real (a_n)}, \qfor n\isnatural.$$ It follows that 
\begin{equation}
\label{f2}
\frac{\sigma^4}{4}\left|w\frac{f'(w)}{f(w)}\right| \leq r \left(\frac{q\sigma^4}{4r} + \sigma^2\sum_{n=1}^\infty \frac{1}{2r + \Real (a_n)}\right).
\end{equation}
Equation (\ref{logderivdunchangemuch}) follows from (\ref{f1}) and (\ref{f2}).

Next we prove (\ref{logderivisnice}). Suppose that $r \geq r_1$. It follows from (\ref{aneweq}), with $z = r$, that
\begin{equation}
\label{e1}
\left|r\frac{f'(r)}{f(r)}\right| = \left|q + \sum_{n=1}^\infty \frac{r}{(r + a_n)}\right| \geq q + \sum_{n=1}^\infty \Real \left(\frac{r}{r+a_n}\right) \geq q + \sigma^2 \sum_{n=1}^\infty \frac{r}{r+|a_n|}.
\end{equation}
Suppose that $w$ is such that $|w| = r$ and $M(r,f) = |f(w)|$. Then 
\begin{equation}
\label{e2}
\frac{\log M(r,f)}{\log r} = \frac{\log|f(w)|}{\log r} \leq \frac{\log |c|}{\log r} + q + \frac{\sum_{n=1}^\infty \log \left(1 + r/|a_n|\right)}{\log r}.
\end{equation}
Let $N_1\isnatural$ be such that $|a_n| \geq 1$, for $n\geq N_1$. 
We claim that 
\begin{equation}
\label{claimeq}
\frac{r}{r+|a_n|} \geq \frac{1}{4}\frac{\log \left(1 + r/|a_n|\right)}{\log r}, \qfor n\geq N_1.
\end{equation}
To prove this claim, suppose that $n\geq N_1$. We consider first the case that $2|a_n| < r$. In this case
\begin{align*}
\frac{r\log r}{r+|a_n|} &> \frac{1}{2}\log r = \frac{1}{4} \log r^2 
                        \geq \frac{1}{4} \log \frac{r^2}{|a_n|^2} \geq \frac{1}{4} \log \left(1 + r/|a_n|\right).
\end{align*}
This completes the proof of the claim in this case. We consider next the case that $2|a_n| \geq r$. It is readily seen by differentiation that the function $$h(x) = (1+x)\log \left(1+\frac{1}{x}\right), \qfor x >0,$$ is decreasing. We deduce that $$\log r \geq \log r_1 \geq \frac{3}{2} \log 3 \geq (1 + x) \log \left(1 + \frac{1}{x}\right), \qfor x \geq \frac{1}{2}.$$ It follows that $$\log r \geq \left(1 + \frac{|a_n|}{r}\right) \log \left(1 + \frac{r}{|a_n|}\right).$$ This completes the proof of our claim (\ref{claimeq}). 

By (\ref{claimeq}), we deduce from (\ref{e1}) and (\ref{e2}) that
\begin{align}
r\left|\frac{f'(r)}{f(r)}\right| - \frac{\sigma^2}{8} \frac{\log M(r,f)}{\log r} &\geq  \frac{\sigma^2}{8} \left(\sum_{n=N_1}^\infty\frac{\log(1 + r/|a_n|)}{\log r} - \frac{\log|c|}{\log r} - \sum_{n=1}^{N_1-1} \frac{\log(1 + r/|a_n|)}{\log r}\right) \nonumber \\
\label{qwe}
&\geq \frac{\sigma^2}{8} \left(\sum_{n=N_1}^\infty\frac{\log(1 + r/|a_n|)}{\log r} - \frac{\log|c|}{\log r} - N_1 \frac{\log(1 + r/|a_1|)}{\log r}\right).
\end{align} 
The first term on the right-hand side of (\ref{qwe}) tends to infinity, as $r$ tends to infinity; however, the other terms are bounded. Hence, for a sufficiently large choice of $r_1$, this is sufficient to establish (\ref{logderivisnice}).

Finally (\ref{lastlogderiv}) follows from (\ref{logderivdunchangemuch}) with $w = r$, and from (\ref{logderivisnice}).
%
%
%
%
%%%%%%%%%%%%%%%%
%
%%%%%%%%%%%%%%%%
%
%
%
%
\section{Proof of Lemma~\ref{domainslemma}}
\label{Sdomains}
We use the following version of the Ahlfors five islands theorem; see, for example, \cite[Theorem 6.2]{MR0164038}.
\begin{lemma}
\label{constlemma}
Suppose that $D_\kappa$, $\kappa\in \{1, 2, 3\}$, are Jordan domains with pairwise disjoint closures. Then there exists $\nu > 0$ with the following property. If $r > 0$ and $f : B(a,r) \to \mathbb{C}$ is an analytic function such that $$\frac{|f'(a)|}{1 + |f(a)|^2} \geq \frac{\nu}{r},$$ then there exists $\kappa\in\{1, 2, 3\}$ such that $B(a,r)$ has a subdomain which is mapped bijectively onto $D_\kappa$.
\end{lemma}
%
%\begin{proof}[Proof of Lemma~\ref{domainslemma}]
To prove Lemma~\ref{domainslemma} we first fix a value of $\nu>0$ such that the conclusion of Lemma~\ref{constlemma} is satisfied for the domains defined in (\ref{thePk}). 

Since, by (\ref{logderivisnice}), $f'(r) \ne 0$, for $r \geq r_1$, we may define a function 
\begin{equation}
\label{trdef}
t(r) = \frac{8\nu}{\sigma^4} \left|\frac{f(r)}{f'(r)}\right|, \qfor r\geq r_1.
\end{equation}
We observe by (\ref{logderivisnice}) that
\begin{equation}
\label{teq}
t(r) \leq \frac{64\nu \log r}{\sigma^6 \log M(r,f)}r, \qfor r \geq r_1.
\end{equation}
It follows from (\ref{itgrows}) and (\ref{teq}) that we may assume that $t(r)$ is small compared to $r$, provided that $r$ is sufficiently large. We deduce that there exist $r_2 \geq r_1$ and $c_1 > 0$ such that if $r\geq r_2$, then $T(r)$ contains at least 
\begin{equation}
\label{mrdef}
m(r) = \frac{c_1 r^2}{t(r)^2}
\end{equation}
disjoint discs of radius $3t(r)$. We may assume that $m(r)$ is large, for $r \geq r_2$.

Suppose that $r \geq r_2$. We choose a sequence of points $(\beta_{q,r})$, for $1\leq q \leq m(r)$, such that the discs $B(\beta_{q,r}, 3t(r))$ are pairwise disjoint, and such that $$B(\beta_{q,r}, 3t(r)) \subset T(r), \qfor 1 \leq q \leq m(r).$$

Choose one of the points $(\beta_{q,r})$, which we denote by $\beta$. We also write $t$ for $t(r)$. We construct a point $b$, a domain $U$, and a compact set $V$ such that, with $b_{q,r} = b$, $U_{q,r} = U$ and $V_{q,r} = V$, the conclusions of the lemma are satisfied.

First we note that $f(z)\ne 0$, for $z \in B(\beta, 3t)$, by (\ref{comparabletoM}), and so we may define a branch of $\log f$ in this disc.

Let $h_\beta : B(\beta, t) \to \mathbb{C}$ be defined by $h_\beta(z) = \log f(z) - \log f(\beta)$. We have $h_\beta(\beta) = 0$. Hence, by (\ref{logderivdunchangemuch}) with $z = \beta$ and $w = r$, by (\ref{trdef}), and since $|\beta| \leq 2r$, $$\frac{|h_\beta'(\beta)|}{1 + |h_\beta(\beta)|^2} = |h_\beta'(\beta)| = \left|\frac{f'(\beta)}{f(\beta)}\right| \geq \frac{\sigma^4}{8} \left| \frac{f'(r)}{f(r)}\right| = \frac{\nu}{t}.$$

Hence, by Lemma~\ref{constlemma}, there is a subdomain of $B(\beta, t)$ which is mapped bijectively by $h_\beta$ onto $P_\kappa$, for some $\kappa\in\{1,2,3\}$, where $P_\kappa$ is as defined in (\ref{thePk}). In particular, there exists $b~\in~B(\beta, t)$ such that $f(b) > 0$. 

We next let $h_{b} : B(b, t) \to \mathbb{C}$ be defined by $h_{b}(z) = \log f(z) - \log f (b)$. We have $h_{b}(b)$ = 0 and so, as above, $$\frac{|h_{b}'(b)|}{1 + |h_{b}(b)|^2} \geq \frac{\nu}{t}.$$

Applying Lemma~\ref{constlemma} a second time, there is a subdomain $U$ of $B(b, t)$ which is mapped bijectively by $h_{b}$ onto $P_{\kappa}$, for some $\kappa\in\{1,2,3\}$. It follows that $\log f$ maps $U$ bijectively onto $\Omega_{\kappa}(b)$. Let $V$ be the subset of $U$ mapped bijectively by $\log f$ onto $Q_{\kappa}(b)$. It follows that $V$ is mapped bijectively by $f$ onto $T(f(b))$. This establishes part (iv) of the lemma.

We note that parts (i) and (ii) of the lemma are immediate, and the set inclusions in part (iii) of the lemma hold by construction.

Finally, we need to estimate the area of $V$. Since $Q_{\kappa}(b) = \log f(V)$, we deduce by (\ref{logderivdunchangemuch}) with $z$ replaced by $r$, by (\ref{trdef}), and since $|w| \geq r$, that 
\begin{align*}
\operatorname{area}(Q_{\kappa}(b)) = 2(\psi - \theta_2)\log 2 &\leq %\sup_{w \in V} |h_{b}'(w)|^2 \operatorname{area} V \\
                       \sup_{w \in V} \left|\frac{f'(w)}{f(w)}\right|^2 \operatorname{area}(V) \\
                       &\leq \left(\frac{4}{\sigma^4} \left|\frac{f'(r)}{f(r)} \right| \right)^2 \operatorname{area}(V) \\
                       &=\frac{1024\nu^2}{\sigma^{16} t^2} \operatorname{area}(V). 
\end{align*}                                  
Hence $\operatorname{area}(V) \geq c_2 t^2$, where $$c_2 = \frac{(\psi - \theta_2)\sigma^{16}\log 2}{512\nu^2}.$$
%
%%%%%%%%%%%%%%%
%
%%%%%%%%%%%%%%%
%
\section{Proof of Lemma~\ref{Elemma}}
\label{buildsets}
In this section we construct the sequence of finite collections of sets $(\mathcal{E}_n)_{n\isnatural}$ to which Lemma~\ref{mcmlemma} is applied. Our construction is very similar to that of Bergweiler and Karpi{\'n}ska in \cite{MR2609307}. The main difference, apart from minor changes in notation, is that their construction is within a whole annulus whereas ours is within a closed sector of an annulus $T(r)$, for $r \geq r_2$; this makes some parts of our construction slightly easier.

We require the following \cite[Lemma 4.1]{MR2609307}.
\begin{lemma}
\label{distlemma}
Suppose that $\Omega$ is a domain, and that $Q \subset \Omega$ is compact. Then there exists $C=C(\Omega, Q)>0$ such that if $g$ is analytic and univalent in $\Omega$, then $$|g'(w)| \leq C|g'(z)|, \qfor w, z \in Q.$$
\end{lemma}
In particular, recalling the definitions given in (\ref{theOmegak}) and (\ref{theQk}), we deduce the following.
\begin{lemma}
\label{distlemma2}
There exists $C>0$ such that if  $\kappa\in\{1,2,3\}$, $a$ is such that $f(a) > 0$ and $g$ is analytic and univalent in $\Omega_\kappa(a)$, then $$|g'(w)| \leq C|g'(z)|, \qfor w,z \in Q_\kappa(a).$$
\end{lemma}
\begin{proof}
This follows from Lemma~\ref{distlemma}, since the sets $\Omega_\kappa(a)$ and $Q_\kappa(a)$ are all translations of fixed sets. 
\end{proof}

We also, for convenience, use the following version of the Koebe distortion theorem; see, for example, \cite[Lemma 4.2]{MR2609307}
\begin{lemma}
\label{kdistlemma}
Suppose that $r>0$ and that $g$ is analytic and univalent in the disc $B(a, r)$. Then $$|g'(z)| \leq 12|g'(a)|, \qfor z \in \overline{B(a, r/2)}.$$
\end{lemma}

We require one other result. 
\begin{lemma}
\label{Lstrangecondition}
Suppose that $f$ is a {\tef}. Then there exists $R_0=R_0(f)>0$ such that
\begin{equation*}
\frac{\log\log M^n(R_0, f)}{n} \rightarrow\infty, \quad\text{as } n\rightarrow\infty.
\end{equation*}
\end{lemma}
\begin{proof}
This result follows immediately from the fact noted in \cite[Equation (6)]{MR1684251} that, for sufficiently large values of $R$, $$\frac{\log\log M(R, f^n)}{n} \rightarrow\infty, \quad\text{as } n\rightarrow\infty.$$ 
\end{proof}
We now prove Lemma~\ref{Elemma}.
We first recall the various sets and points constructed in the proof of Lemma~\ref{domainslemma}, in particular, for $r \geq r_2$, the points $b_{q, r}$, domains $U_{q, r}$ and connected compact sets $V_{q, r}$, for $1 \leq q \leq m(r)$.

It is well-known (see, for example, \cite{Rippon01102012}) that if $k > 1$, then 
\begin{equation*}
\frac{M(kr,f)}{M(r,f)} \rightarrow\infty \text{ as } r\rightarrow\infty.
\end{equation*}
It follows that we may choose $R_1 \geq R_0$, where $R_0$ is as defined in the statement of Lemma~\ref{Lstrangecondition}, sufficiently large that
\begin{equation}
\label{R0def}
M(r,f) \geq \max\left\{2, \frac{1}{\sigma^2}\right\} M(\sigma r,f), \qfor r \geq R_1. 
\end{equation}
Choose $\rho_0>0$ sufficiently large that $\sigma^2\rho_0 \geq \max\{r_2, R_1\}$. We claim that
\begin{equation}
\label{rhobig}
\mu^k(\rho_0) \geq \max\left\{2, \frac{1}{\sigma^2}\right\} M^k(\sigma^2\rho_0) \geq 2 M^k(R_1), \qfor k\isnatural.
\end{equation}

The right-hand inequality of (\ref{rhobig}) is obvious. The left-hand inequality of (\ref{rhobig}) is obtained by induction. It is clearly true for $k=1$, and when $k=n+1$ we have
\begin{align*}
\mu^{n+1}(\rho_0) &\geq \mu\left(\frac{1}{\sigma^2} M^n(\sigma^2 \rho_0)\right) &\text{inductively by (\ref{rhobig})} \\
                  &\geq \max\left\{2, \frac{1}{\sigma^2}\right\} M^{n+1}(\sigma^2 \rho_0) &\text{by (\ref{mugrows}) and (\ref{R0def})}.
\end{align*}

Put $$\mathcal{E}_0 = \{ T(\rho_0) \} \quad\text{ and }\quad \mathcal{E}_1 = \{V_{q,\rho_0} : 1 \leq q \leq m(\rho_0) \}.$$

We claim that we may define collections of compact sets $(\mathcal{E}_n)_{n\in\{0,1,\ldots\}}$ such that if $F~\in~\mathcal{E}_{n+1}$, then the following hold.
\begin{itemize}
\item The function $f^{n+1} : F \to T(\rho_{n+1,F})$, for some $\rho_{n+1,F} > \rho_0$, is a bijection. 
\item If $G \in \mathcal{E}_{n}$ is such that $F \subset G$, then there exists $q~\in~\{ 1,\cdots,m(\rho_{n,G})\}$ such that $f^{n}(F) = V_{q,\rho_{n,G}} \subset T(\rho_{n,G})$.
\item The sequence $(\mathcal{E}_n)_{n\isnatural}$ has the properties (i) and (ii) discussed before the statement of Lemma~\ref{mcmlemma}.
\end{itemize}

We establish this claim inductively. First we note that if $n=0$, then the claim follows straightforwardly from Lemma~\ref{domainslemma}. Suppose then that $n\isnatural$ and that $G \in \mathcal{E}_n$. For simplicity we write $\rho$ for $\rho_{n,G}$. See Figure~\ref{fig.x} for a simple picture of the sets used in this proof. 
Since, by assumption, $f^{n} : G \to T(\rho)$ is a bijection, we may let $h$ be the branch of the inverse of $f^n$ which maps $T(\rho)$ to $G$. We then set $$\mathcal{E}_{n+1}(G) = \{ h(V_{q,\rho}) : 1 \leq q \leq m(\rho) \}.$$ Finally we set $$\mathcal{E}_{n+1} = \bigcup_{G \in \mathcal{E}_n} \mathcal{E}_{n+1}(G).$$

The claimed properties now follow by construction, and by Lemma~\ref{domainslemma}. Note that (\ref{Eisnice}) also follows from the construction.

\begin{figure}[ht]
	\centering
	\includegraphics[width=10cm,height=7cm]{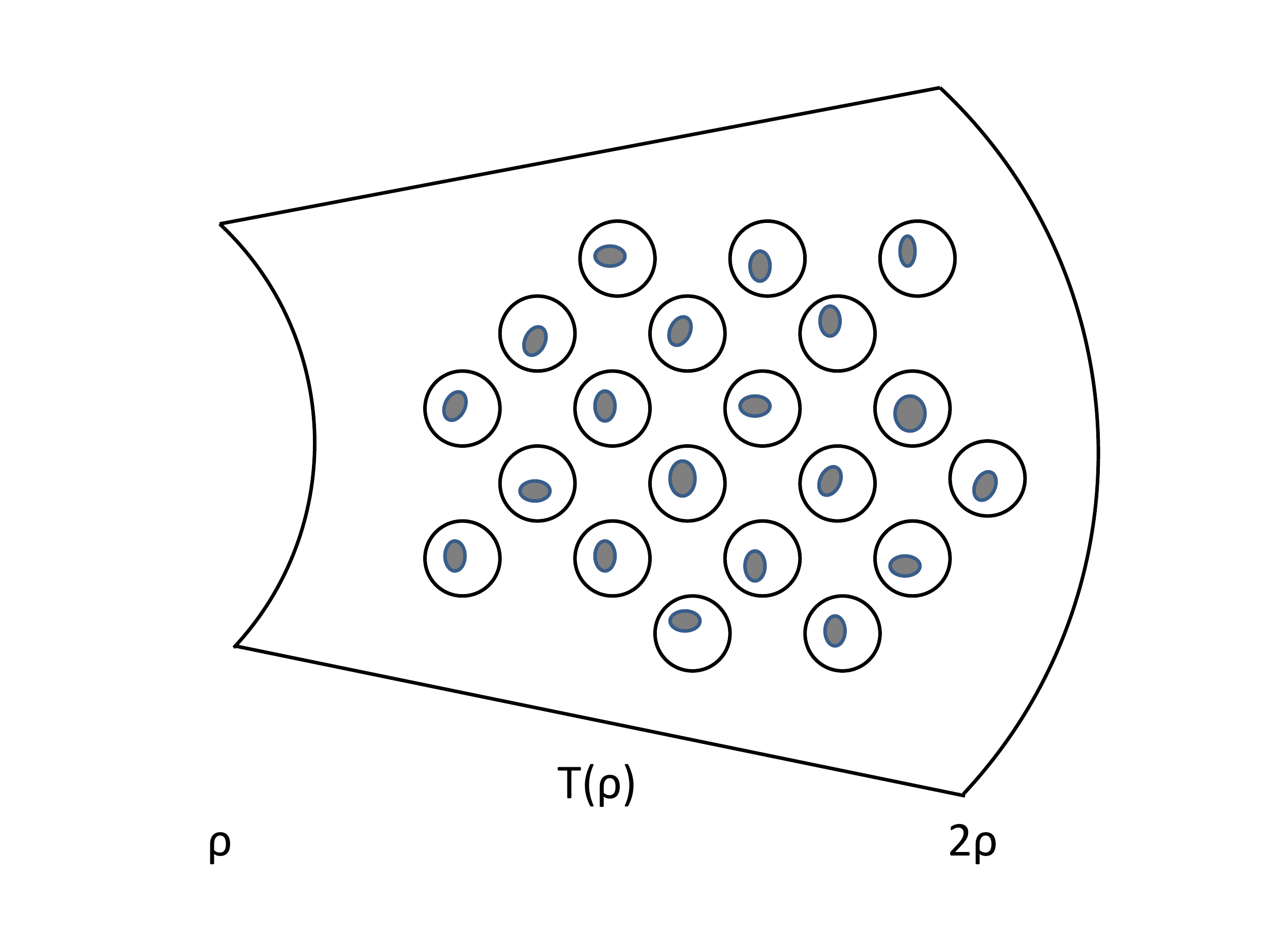}
	\caption{$f^{n}(G) = T(\rho)$, where $G \in \mathcal{E}_{n}$, containing $m(\rho)$ balls of radius $t(\rho)$, each of which contains one of the $V_{q, \rho}$, shown shaded. Note that, by construction, each ball is a distance of at least $t(\rho)$ from the boundary of $T(\rho)$. The preimages under $f^n$ of the $V_{q, \rho}$ make up the sets $\mathcal{E}_{n+1}(G)$.}
  \label{fig.x}
\end{figure}  

Fix $F \in \mathcal{E}_n$, for some $n\isnatural$, and set $\rho = \rho_{n, F}$ for simplicity. We need to establish (\ref{densityofsets}), which concerns densities of sets, and (\ref{diamofsets}) and (\ref{dneq}), which concern the diameters of sets.

First we prove (\ref{densityofsets}). We deduce from Lemma~\ref{domainslemma} that $T(\rho) = \exp W$, where $W = Q_\kappa(a)$ for some $\kappa\in\{1,2,3\}$ and some $a$ such that $f(a) = \rho$. Moreover, the function $\log f^n : F \to W$ is a bijection, and the inverse branch of $\log f^n$, $h : W \to F$, extends to $\Omega_\kappa(a)$.

Suppose that $1 \leq q \leq m(\rho)$. Let $W_q \subset W$ be such that $\exp W_q = V_{q, \rho}$. Since $|z| \leq 2\rho$, for $z \in V_{q, \rho}$, we deduce that  $$\operatorname{area}(W_q) = \int_{V_{q, \rho}} \frac{1}{|z|^2} \ dx \ dy \geq \frac{1}{4\rho^2} \operatorname{area}(V_{q, \rho}).$$ It follows by Lemma~\ref{domainslemma}, and by (\ref{mrdef}), that $$\operatorname{area}\left(\bigcup_{q=1}^{m(\rho)} W_q\right) \geq \frac{c_2 t(\rho)^2}{4\rho^2} m(\rho)  = \frac{c_1 c_2}{4}.$$ Hence, by Lemma~\ref{distlemma2}, 
\begin{align*}
\operatorname{dens}(E_{n+1}, F) &= \operatorname{dens}\left(\bigcup_{q=1}^{m(\rho)} h(W_q), h(W)\right) \\
                                &\geq \frac{1}{C^2} \operatorname{dens}\left(\bigcup_{q=1}^{m(\rho)} W_q, W\right) \\
                                &\geq \frac{c_1 c_2}{4C^2 \operatorname{area}(W)} \\
                                &= \frac{c_1 c_2}{8C^2 (\psi - \theta_2)\log 2},
\end{align*}                     
where $C$ is the constant in Lemma~\ref{distlemma2}. This completes the proof of (\ref{densityofsets}), with $$c_3 = \frac{c_1 c_2}{8C^2 (\psi - \theta_2)\log 2}.$$

Finally we prove (\ref{diamofsets}) and (\ref{dneq}). We may assume that $n\geq 2$. Suppose that $F_m \in \mathcal{E}_m$ is such that $F \subset F_m$, for $0 \leq m < n$. For simplicity we write $\rho_m$ for $\rho_{m, F_m}$. 

It follows from Lemma~\ref{domainslemma} that $f^{n-1}$ maps $F_{n-1}$ bijectively to $T(\rho_{n-1})$, and that 
\begin{equation*}
f^{n-1}(F) = V_{q, \rho_{n-1}} \subset B(b_{q},t(\rho_{n-1})) \subset B(b_{q},2t(\rho_{n-1})) \subset T(\rho_{n-1}),
\end{equation*}
for some $q \in \{ 1,\cdots,m(\rho_{n-1})\}$. 

We let $h$ be the branch of the inverse of $f^{n-1}$ which maps $f^{n-1}(F)$ to $F$, and note that $h$ may be extended to a function which is univalent in $B(b_{q},2t(\rho_{n-1}))$. Hence, by Lemma~\ref{kdistlemma}, we have that $$|h'(z)| \leq 12 |h'(b_q)|, \qfor z \in B(b_{q},t(\rho_{n-1})).$$ It follows that 
\begin{equation}
\label{Fdiam}
\operatorname{diam} F \leq 12|h'(b_q)| \operatorname{diam} f^{n-1}(F) \leq 24 |h'(b_q)| t(\rho_{n-1}).
\end{equation}

Now, let $\zeta = h(b_q)$. It is well-known that $$\frac{\log M(r, f)}{\log r}$$ is an increasing function of $r$. Since $f^m(\zeta) \in T(\rho_m)$, for $0 \leq m < n$, it follows from (\ref{lastlogderiv}) that $$|f'(f^m(\zeta))| \geq c_4 \frac{\log M(\rho_m,f)}{\log \rho_m} \frac{\rho_{m+1}}{\rho_m}, \qfor 0 \leq m < n-1,$$ 
where $c_4 = \sigma^6/64.$

It follows by the chain rule that 
\begin{equation}
\label{heq}
|h'(b_q)| \leq \frac{\rho_0}{\rho_{n-1}} \prod_{m=0}^{n-2} \frac{\log \rho_m}{c_4 \log M(\rho_m, f)}.
\end{equation}

Now, by (\ref{logderivisnice}) and (\ref{trdef}), 
\begin{equation}
\label{omeq}
\frac{\log M(\rho_{n-1},f)}{\log \rho_{n-1}} \leq \frac{8\rho_{n-1}}{\sigma^2}\left|\frac{f'(\rho_{n-1})}{f(\rho_{n-1})}\right| = \frac{64\nu\rho_{n-1}}{\sigma^6 t(\rho_{n-1})}.
\end{equation}

We deduce from (\ref{Fdiam}), (\ref{heq}) and (\ref{omeq}) that $$\operatorname{diam} F \leq \frac{24 \rho_0}{\rho_{n-1}} \prod_{m=0}^{n-2} \frac{\log \rho_m}{c_4 \log M(\rho_m, f)} t(\rho_{n-1}) \leq \frac{c_5}{c_4^{n-1}}\prod_{m=0}^{n-1} \frac{\log \rho_m}{\log M(\rho_m, f)},$$ where $c_5 = 1536\nu\rho_0/\sigma^6$.

We recall the definition of the function $\mu$ from (\ref{mudef}), and note that, by (\ref{comparabletoM}),
\begin{equation}
\label{rhogrow}
2\rho_m \geq |f^m(\zeta)| \geq \mu^m(\rho_0), \qfor 0 \leq m < n.
\end{equation}
It follows from (\ref{rhobig}) and (\ref{rhogrow}) that $\rho_m \geq M^m(R_1)$ for $0 \leq m < n$, where $R_1$ is as in (\ref{R0def}). We deduce that $$\frac{\log M(\rho_m, f)}{\log \rho_m} \geq \frac{\log M^{m+1}(R_1)}{\log M^m(R_1)}, \qfor 0 \leq m < n.$$
It follows that 
\begin{align*}
\operatorname{diam} F &\leq c_5 c_4^{1-n} \frac{\log R_1}{\log M^{n}(R_1)} = \exp\left(c_6 - n\log c_4 - \log \log M^n(R_1)\right),  
\end{align*}                     
where $c_6 = \log \ (c_4c_5\log R_1)$. 

We set $d_n = \exp\left(c_6 - n\log c_4 - \log \log M^n(R_1)\right)$, for $n\isnatural$. By Lemma~\ref{Lstrangecondition},
\begin{align*}
\lim_{n\rightarrow\infty} \frac{n}{|\log d_n|} &= \lim_{n\rightarrow\infty} \frac{n}{|n\log c_4 + \log\log M^n(R_1) - c_6|} = 0,
\end{align*}
and so (\ref{dneq}) is satisfied. It also follows that $d_n\rightarrow 0$ as $n\rightarrow\infty$. This completes the proof of the lemma.\\

%
% Ack
%
\itshape Acknowledgment: \normalfont
The author is grateful to Gwyneth Stallard and Phil Rippon for all their help with this paper.
%
%%%%%%%%%%%%%

%%%%%%%%%%%%%
%
% BIBLIOGRAPHY
%
%%%%%%%%%%%%%
\bibliographystyle{acm}
\bibliography{../Research.References}
\end{document}